\documentclass[11pt,a4paper]{amsart}
\usepackage{graphicx,multirow,array,amsmath,amssymb,upgreek}

\newtheorem{theorem}{Theorem}

\newtheorem{proposition}[theorem]{Proposition}

\begin{document}

\title{Circuit presentation and lattice stick number with exactly 4 $z$-sticks}
\author[H. Kim]{Hyoungjun Kim}
\address{Institute of Mathematical Sciences, Ewha Womans University, 52, Ewhayeodae-gil, Seodaemun-gu,
 	Seoul 03760, Korea}
\email{kimhjun@korea.ac.kr}
\author[S. No]{Sungjong No}
\address{Institute of Mathematical Sciences, Ewha Womans University, 52, Ewhayeodae-gil, Seodaemun-gu,
 	Seoul 03760, Korea}
\email{sungjongno84@gmail.com}

\thanks{The first author(Hyoungjun Kim) was supported by Basic Science Research Program through the National Research Foundation of Korea (NRF) funded by the Korea government Ministry of Education(2009-0093827) and Ministry of Science and ICT(NRF-2018R1C1B6006692).}
\thanks{The corresponding author(Sungjong No) was supported by Basic Science Research Program through the National Research Foundation of Korea (NRF) funded by the Korea government Ministry of Education(2009-0093827).}

\begin{abstract}
The lattice stick number $s_L(L)$ of a link $L$ is defined to be the minimal number of straight line segments required to construct a stick presentation of $L$ in the cubic lattice.
Hong, No and Oh \cite{HNO1} found a general upper bound $s_L(K) \leq 3 c(K) +2$.
A rational link can be represented by a lattice presentation with exactly 4 $z$-sticks.

An $n$-circuit is the disjoint union of $n$ arcs in the  lattice plane $\mathbb{Z}^2$.
An $n$-circuit presentation is an embedding obtained from the $n$-circuit by connecting each $n$ pair of vertices with one line segment above the circuit.
By using a 2-circuit presentation, we can easily find the lattice presentation with exactly 4 $z$-sticks.

In this paper, we show that an upper bound for the lattice stick number of rational $\dfrac{p}{q}$-links realized with exactly 4 $z$-sticks is $2p+6$.
Furthermore it is $2p+5$ if $L$ is a 2-component link.
\end{abstract}

\maketitle

\section{Introduction} \label{sec:int}

A {\em link} is a disjoint union of simple closed curves in $\mathbb{R}^3$.
A one component link is called a {\em knot}.
A polygonal link is a link whose image in $\mathbb{R}^3$ is the union of a finite set of line segments, called {\em sticks}.
Let the {\em stick number} $s(L)$ of a link $L$ be the least number of straight sticks need to make a link $L$.
In particular, stick number of the trefoil knot $3_1$ is equal to 6.

\begin{figure}[h!]
\includegraphics[scale=0.4]{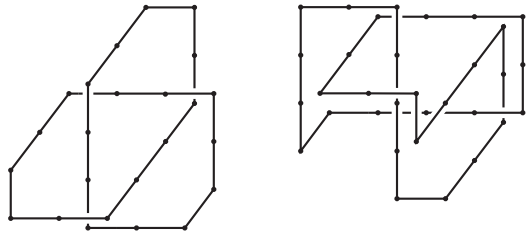}
\caption{Trefoil knot and figure-8 knot in $\mathbb{Z}^3$}
\label{fig:lat3141}
\end{figure}

Let the {\em lattice stick number} $s_L(L)$ of a link $L$ be the least number of sticks need to make a link $L$ in the cubic lattice $\mathbb{Z}^3$ which is $(\mathbb{R} \times \mathbb{Z} \times \mathbb{Z}) \cup (\mathbb{Z} \times \mathbb{R} \times \mathbb{Z}) \cup (\mathbb{Z} \times \mathbb{Z} \times \mathbb{R})$.
Huh and Oh~\cite{{HO1},{HO2}} showed that there are exactly two non-trivial knots whose lattice stick numbers are at most 14.
More precisely, the trefoil knot $3_1$ has lattice stick number 12, and the figure-8 knot $4_1$ has lattice stick numbers 14, which are depicted in Figure~\ref{fig:lat3141}.
Huang and Yang~\cite{HY} showed that $5_1$ and $5_2$ knot have lattice stick number 16.
Adams et al.~\cite{ACCJSZ} showed that $s_L(8_{20}) = s_L(8_{21}) = s_L(9_{46}) = 18$, $s_L(4^2_1) = 13$ and $s_L(T_{p,p+1}) = 6p$ for $p \geq 2$ where $T_{p,p+1}$ is a $(p,p+1)$-torus knot.
To find the exact values of the lattice stick number of these knots, they used a lower bound on lattice stick number in terms of bridge number, $s_L(K) \geq 6 b(K)$, which was proved by Janse van Rensburg and Promislow~\cite{JP}.
Furthermore Diao and Ernst \cite{DE} found a general lower bound in terms of crossing number $c(K)$
which is $s_L(K) \geq 3 \sqrt{c(K)+1} +3$ for a nontrivial knot $K$.
Hong, No and Oh \cite{HNO1} found a general upper bound $s_L(K) \leq 3 c(K) +2$, and moreover $s_L(K) \leq 3 c(K) - 4$ for a non-alternating prime knot.
They~\cite{HNO2} also showed that $s_L(2^2_1)=8$, $s_L(2^2_1 \sharp 2^2_1)=s_L(6^3_2)=s_L(6^3_3)=12$, $s_L(4^2_1)=13$, $s_L(5^2_1)=14$ and any other non-split links have stick numbers at least 15.

Henceforth, a stick in $L$ which is parallel to the $x$-axis is called an {\em $x$-stick\/}.
An {\em $x$-level\/} $k$ for some integer $k$ is a $yz$-plane whose $x$-coordinate is $k$.
Similarly define for $y$-stick or $z$-stick, and $y$-level or $z$-level.
Then each $y$-stick and $z$-stick lies entirely on an $x$-level.
It is known that the knot in the cubic lattice must have at least 4 $z$-sticks to be a nontrivial knot.
Note that the only 2-bridge links and the trivial knot can be represented using exactly 4 $z$-sticks.
So we study the representations of 2-bridge links using 4 $z$-sticks.

An  {\em $n$-circuit} is the disjoint union of $n$ arcs in lattice plane $\mathbb{Z}^2$ such that every endpoint of arcs is labeled $v_1,v'_1,\dots,v_n,v'_n$.
Note that the coordinate of $v_i$ is $(x_i,y_i)$ and the coordinate of $v'_i$ is $(x'_i,y'_i)$.
A circuit is called {\em regular} when it satisfies the following two conditions;

For each $i,j \in \{ 1,\dots,n \}$,
\begin{itemize}
\item $x_i=x'_i$ or $y_i=y'_i$,
\item $[x_i,x'_i] \cap [x_j,x'_j] = \phi$ or $[y_i,y'_i] \cap [y_j,y'_j] = \phi$, if $i \ne j$.
\end{itemize}
An {\em $n$-circuit presentation} is an embedding obtained from the $n$-regular circuit by connecting each $n$ pair of vertices $v_i$ and $v'_i$ with one line segment above the circuit.
Figure \ref{fig:ecp} is an example of a 3-circuit presentation.

\begin{figure}[h!]
\includegraphics[scale=0.75]{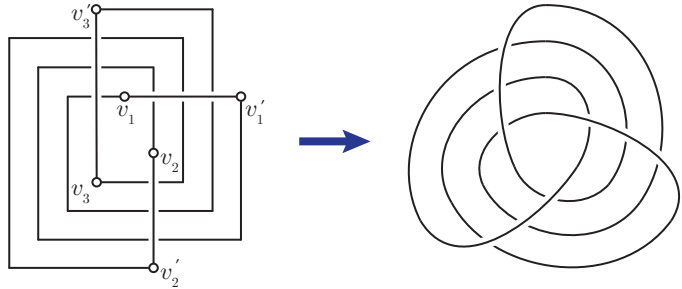}
\caption{Examples of a 3-circuit presentation of $(4,3)$-torus knot}
\label{fig:ecp}
\end{figure}

In this paper, we show the following theorem by using 2-circuit presentations.

\begin{theorem}\label{thm:main}
An upper bound for the lattice stick number of rational $\dfrac{p}{q}$-links realized with exactly 4 $z$-sticks is $2p+6$.
Furthermore it is $2p+5$ if $L$ is a 2-component link.
\end{theorem}

Note that, $3_1$ knot and $5_1$ knot are rational $\frac{3}{1}$-knot and $\frac{5}{1}$-knot, respectively.
The upper bound of the main theorem for these knot are 12 and 16, respectively.
These values are exactly same with their lattice stick number.
In addition, the upper bound of $4^2_1$ link is 13, and this value is also exactly same with its lattice stick number.

The upper bound of the main theorem is the number of sticks to represent a lattice knot with exactly 4 $z$-sticks.
So the general upper bound of lattice stick number is better than our result.
However, for some cases, our result is better.
We consider the general upper bound of the lattice stick number for the rational $(p,1)$-link.
The upper bound of the main theorem of a rational $(p,1)$-link is $2p+6$.
Since the crossing number of a rational $(p,1)$-link is $p$, the general upper bound of lattice stick number of this link is $3p+2$.
Therefore, for a rational $(p,1)$-link, the result of the main theorem is better than the existing result when $p > 4$.

\section{Rational tangle in pillowcase form}\label{sec:cp}

An {\em $n$-tangle} is a proper embedding of the disjoint union of $n$ arcs into a 3-ball.
Note that embedding must send the endpoints of the arcs on the boundary of the ball.
Two $n$-tangles are {\em equivalent} if there is an ambient isotopy of one tangle to the other keeping the boundary of the ball fixed.

\begin{figure}[h!]
\includegraphics[scale=0.6]{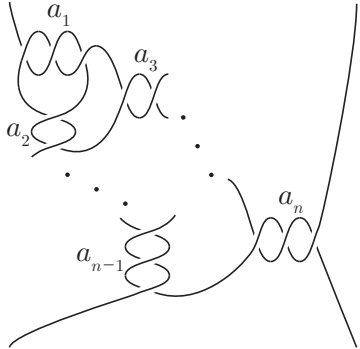}
\caption{The rational tangle with $T(a_1,\dots,a_n)$}
\label{fig:rat}
\end{figure}

A 2-tangle is called a rational tangle with Conway notation $T(a_1,\dots,a_n)$ as drawn in figure~\ref{fig:rat}.
It is called a rational $\dfrac{p}{q}$-tangle if
$$\cfrac{p}{q} = a_n + \cfrac{1}{a_{n-1} + \cfrac{1}{ \ddots +\cfrac{1}{a_2+ \cfrac{1}{a_1} }}}.$$
Conway's Theorem states that a rational $\dfrac{p}{q}$-tangle is ambient isotopic to a rational $\dfrac{p'}{q'}$-tangle if and only if $\dfrac{p}{q}=\dfrac{p'}{q'}$.
It is first stated in~\cite{Co}.

\begin{figure}[h!]
\includegraphics[scale=1]{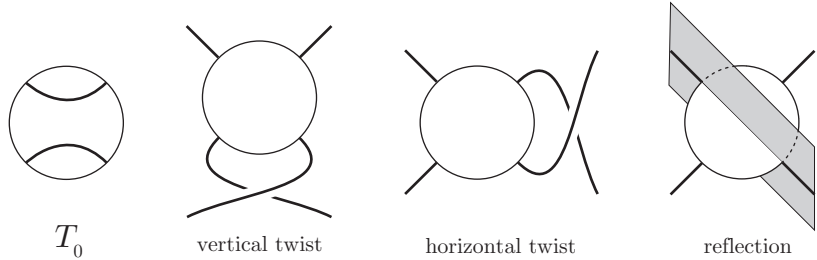}
\caption{Rational tangles and its operations}
\label{fig:twist}
\end{figure}

The {\em vertical twist} is a homeomorphism of a 2-tangle which switches the location of the two endpoints at the bottom by using a half-twist as drawn in Figure~\ref{fig:twist}.
Then the new crossing obtained through this operation or its inverse is called a {\em vertical crossing}.
Similarly, we can also define the {\em horizontal twist} and {\em horizontal crossing}.
It is well known that any positive rational tangle can be obtained from $T_0$ by using a sequence of vertical and horizontal twists(without using their inverses).
If the rational tangle is finished with the vertical crossing, we change it to be finished with the horizontal crossing by using the reflection in the shaded plane as drawn in Figure~\ref{fig:twist}.
A rational $\dfrac{p}{q}$-link is obtained from a rational $\dfrac{p}{q}$-tangle by joining the top endpoints together and the bottom endpoints also together as drawn in figure~\ref{fig:rl}.

\begin{figure}[h!]
\includegraphics[scale=0.75]{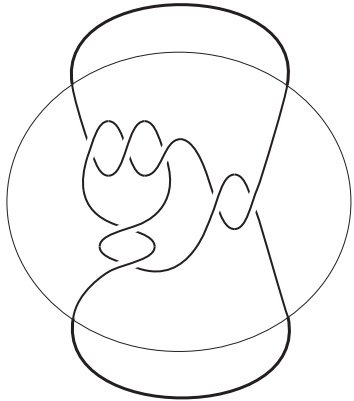}
\caption{A rational $\dfrac{17}{7}$-link}
\label{fig:rl}
\end{figure}

Now we introduce the another description of a rational tangle.
A rational tangle is called the {\em pillowcase form} if the arcs of the tangle lies on its boundary as drawn in figure~\ref{fig:pillow}.
Let $t$(or $s$) be the number of the gaps between the arcs across the top and bottom(or each side, respectively).
Note that $t=0$ means that the arc connecting two top(or bottom) vertices lies on the top(or bottom) boundary.
Similarly, $s=0$ means that the arc connecting two left(or right) vertices lies on the left(or right) side boundary.
Then we call this $(t,s)$-form.

\begin{figure}[h!]
\includegraphics[scale=1]{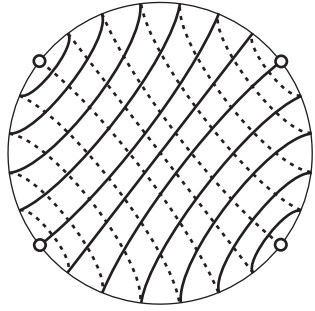}
\caption{The rational tangle in the pillowcase form with $(7,5)$}
\label{fig:pillow}
\end{figure}

For these two different descriptions of a rational tangle, we can obtain Proposition~\ref{prop:cw}.
The result of this proposition is introduced by Cromwell~\cite{C}.

\begin{proposition}\label{prop:cw}
If $(t,s)$-form is equivalent to the rational $\dfrac{p}{q}$-tangle then $(t,s)=(p,q)$.
\end{proposition}

\begin{proof}
We first claim that a $(t,s)$-form turns into the $(t,t+s)$-form when adding a vertical crossing, and turns into the $(t+s,s)$-form when adding a horizontal crossing.
Adding a vertical crossing to a tangle is corresponding to twisting the tangle by $180^{\circ}$ along the vertical crossing.
In this case, the arcs intersect each side exactly $t$ more times.
So the number of the gaps between the arcs across the each side becomes $t+s$.
Similarly, the number of the gaps between the arcs across the top and bottom becomes $t+s$ when adds a horizontal crossing.

Let the rational $\dfrac{p}{q}$-tangle be depicted by Conway notation $T(a_1,\dots,a_n)$.
Now we construct a $(t,s)$-form of the rational $\dfrac{p}{q}$-tangle.
Similar to the construction of the rational tangle, we also construct the pillowcase form from $T_0$ by using a sequence of vertical and horizontal twists.
Note that $T_0$ is $(0,1)$-form.
First add $a_1$ horizontal crossings to $T_0$ then $(0,1)$-form is changed to $(a_1,1)$-form.
Next add $a_2$ vertical crossings then it is changed to $(a_1,a_1 a_2 + 1)$-form.
Repeat this operation for the $n$-th time.
Let a pillowcase form which represents this operation until the $k$-th time be $(t_k,s_k)$-form.

We will show that if $k$ is odd then a $(t_k,s_k)$-form is equivalent to a $\dfrac{t_k}{s_k}$-tangle, and if $k$ is even then the $(s_k,t_k)$-form which is obtained from the $(t_k,s_k)$-form by using the reflection, is equivalent to an $\dfrac{s_k}{t_k}$-tangle.
Now we use the induction on $k$.

When $k=1$, an $(a_1,1)$-form is equivalent to an $a_1$-tangle.
When $k$ is odd, suppose that the $(t_k,s_k)$-form is equivalent to the $\dfrac{t_k}{s_k}$-tangle.
It sufficient to show that the $(s_{k+1},t_{k+1})$-form is equivalent to the $\dfrac{s_{k+1}}{t_{k+1}}$-tangle, since $k+1$ is even.
The $(t_{k+1},s_{k+1})$-form is obtained from the $(t_k,s_k)$-form by adding $a_{k+1}$ vertical crossings.
Then the $(t_{k+1},s_{k+1})$-form is equivalent to the $(t_k,s_k+a_{k+1}t_k)$-form, by the above claim.
On the other hand, a rational tangle with Conway notation $T(a_1,\dots,a_{k+1})$ is an $\left( a_{k+1} + \dfrac{1}{\frac{t_k}{s_k}} \right)$-tangle, so it is an $\left( \dfrac{a_{k+1}t_k+s_k}{t_k} \right)$-tangle.

When $k$ is even, suppose that the $(s_k,t_k)$-form is equivalent to the $\dfrac{s_k}{t_k}$-tangle.
Similar to previous case, it is sufficient to show that the $(t_{k+1},s_{k+1})$-form is equivalent to the $\dfrac{t_{k+1}}{s_{k+1}}$-tangle, since $k+1$ is odd.
The $(t_{k+1},s_{k+1})$-form is obtained from the $(t_k,s_k)$-form by adding $a_{k+1}$ horizontal crossings.
Then the $(t_{k+1},s_{k+1})$-form is equivalent to the $(t_k+a_{k+1}s_k,s_k)$-form, by the above claim.
On the other hand, a rational tangle with Conway notation $T(a_1,\dots,a_{k+1})$ is an $\left( a_{k+1} + \dfrac{1}{\frac{s_k}{t_k}}\right)$-tangle, so it is a $\left( \dfrac{t_k+a_{k+1}s_k}{s_k}\right)$-tangle.

We remark that the $(t,s)$-form is constructed from a rational $\dfrac{p}{q}$-tangle.
If $n$ is odd, the $(t,s)$-form is equivalent to the $(t_n,s_n)$-form, and it is also equivalent to a rational $\dfrac{t_n}{s_n}$-tangle, so it is a $\dfrac{t}{s}$-tangle.
If $n$ is even, by the definition of the rational tangle, a $(t,s)$-form is equivalent to a $(s_n,t_n)$-form.
Then the $(s_n,t_n)$-form is equivalent to a rational $\dfrac{s_n}{t_n}$-tangle, so it is a $\dfrac{t}{s}$-tangle.
In both cases, $p=t$ and $q=s$, by the uniqueness of a rational tangle.
Therefore the $(p,q)$-form is equivalent to the rational $\dfrac{p}{q}$-tangle.
\end{proof}

\section{Proof of main theorem}\label{sec:main}

Let $L$ denote a rational $\dfrac{p}{q}$-link such that $p$ and $q$ are coprime.
By proposition~\ref{prop:cw}, there is an associated diagram as shown in Figure~\ref{fig:slope} (a).
In this figure, the central line segments consists of lines of slope $\pm ~\dfrac{p}{q}$ in the square with four vertices $A(0,0)$, $B(0,-1)$, $C(1,-1)$ and $D(1,0)$.
$l_1$(or $l_2$) is an arc connecting $A$ and $D$(or $B$ and $C$, respectively).
We change this diagram to new diagram as drawn in Figure~\ref{fig:slope} (b), by using the linear transformation
$\begin{pmatrix}
p & -q\\
p & q\\
\end{pmatrix}$.
Then all central line segments are parallel to either the $x$-axis or the $y$-axis.
We say that the line segment is horizontal(or vertical) when it is parallel to the $x$-axis
(or the $y$-axis).

\begin{figure}[h!]
\includegraphics[scale=1]{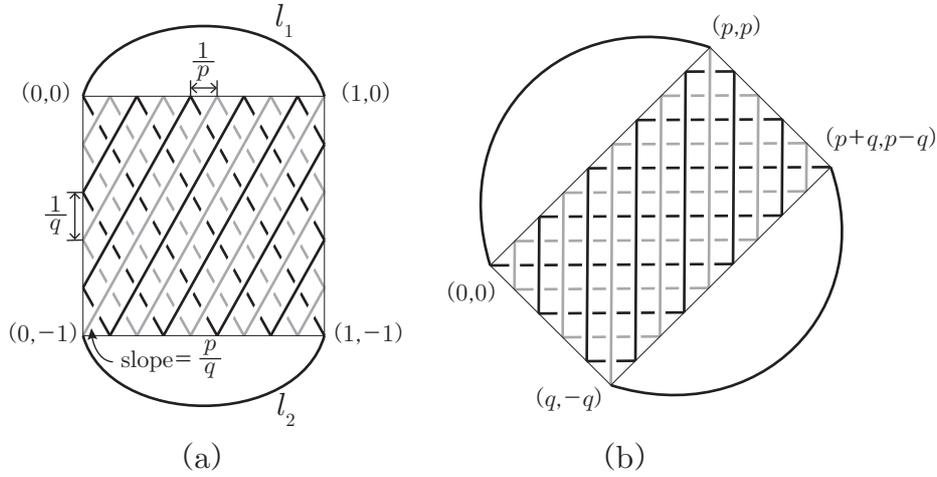}
\caption{A $(p,q)$-form in $\mathbb{Z}^2$}
\label{fig:slope}
\end{figure}

Now, we will make a regular 2-circuit to obtain the link in the cubic lattice with 4 $z$-sticks from the diagram in figure~\ref{fig:slope} (b).
First, transform the diagram for setting into $\mathbb{Z}^2$ except the arcs $l_1$ and $l_2$.
Move the horizontal lines by the following operation.
For each $i \in \{ 1,2,\dots,q \}$, move the line segment between $(p-i,p-i)$ and $(p+i,p-i)$ to the line segment between $(p-i,p+i)$ and $(p+i,p+i)$.
Similarly, for each $i \in \{ 1,2,\dots,q \}$, move the line segment between $(q-i,-q+i)$ and $(q+i,-q+i)$ to the line segment between $(q-i,-q-i)$ and $(q+i,-q-i)$.
During this process, each moved horizontal line segment is separated from its adjacent vertical line segments.
To connect these separated line segments, extend the vertical line segments as drawn in Figure~\ref{fig:circuit} (a).

\begin{figure}[h!]
\includegraphics[scale=0.8]{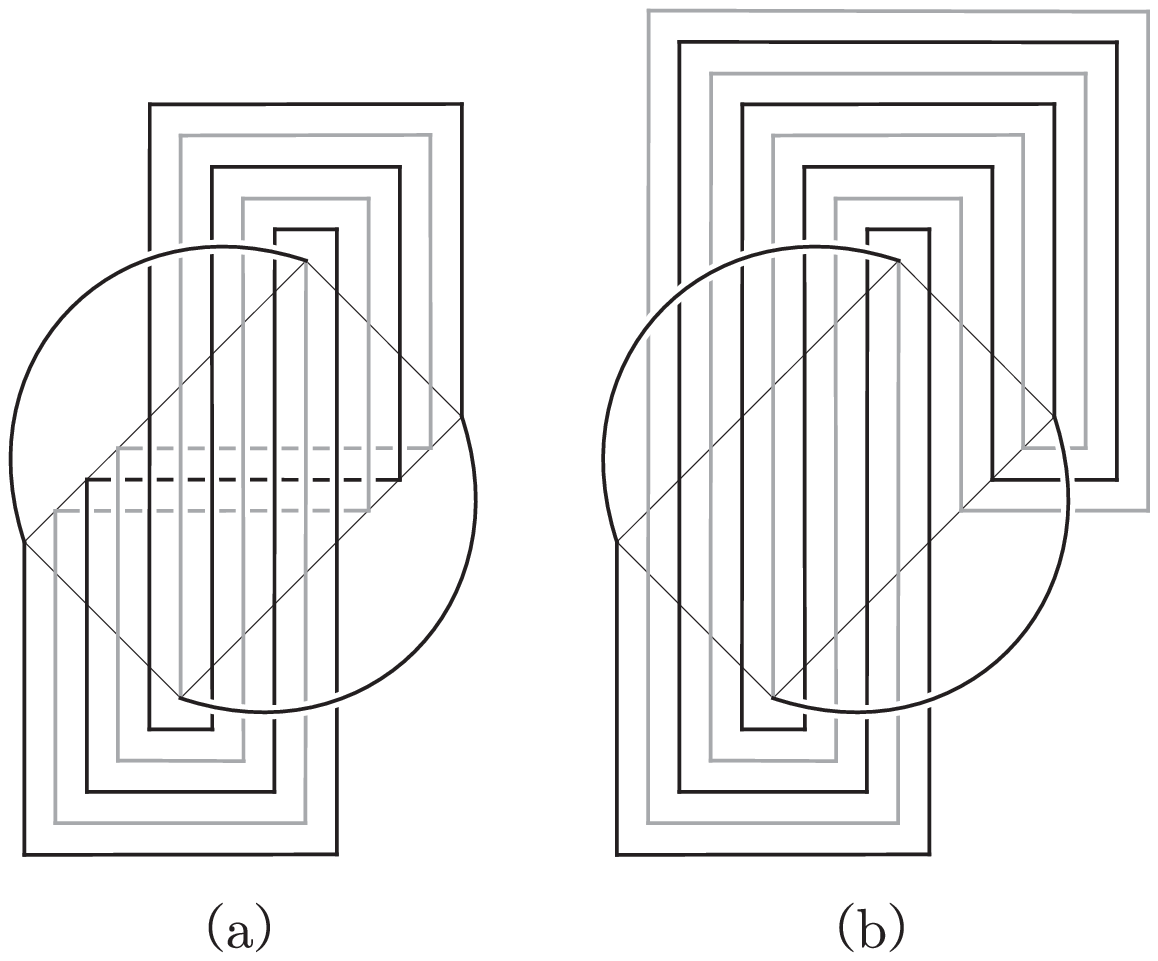}
\caption{How to move the horizontal lines}
\label{fig:circuit}
\end{figure}

We can easily check that there remain $p-q-1$ horizontal line segments that intersect vertical lines.
Since these horizontal line segments can not apply the above process,
change the line segment between $(p-q-i,p-q-i)$ and $(p+q-i,p-q-i)$ to the path which passes through to three points $(p-q-i,p+q+i)$, $(p+q+i,p+q+i)$ and $(p+q+i,p-q-i)$, for each $i \in \{ 1,2,\dots,p-q-1 \}$ as drawn in Figure~\ref{fig:circuit} (b).
When $p - q = 1$ we do not need this operation.

To obtain a regular 2-circuit, we extend the line segment between $(0,0)$ and $(0,-2q)$ to the line segment between $(0,p)$ and $(0,-2q)$.
Similarly, extend the line segment between $(p+q,p-q)$ and $(p+q,p+q)$ to the line segment between $(p+q,-q)$ and $(p+q,p+q)$.
If the extended vertical line segment meets some horizontal line segments then drop these horizontal line segments, and extend the vertical line segments attached to them until they reach the dropped horizontal line segments as drawn in Figure~\ref{fig:regcir} (a).
Let $(0,p)$, $(p,p)$, $(q,-q)$ and $(p+q,-q)$ denote $v_1$, $v'_1$, $v_2$ and $v'_2$ respectively.
Then the resulting diagram is a regular 2-circuit.
Note that there are $2p$ vertical line segments and $2p-2$ horizontal line segments.

\begin{figure}[h!]
\includegraphics[scale=0.6]{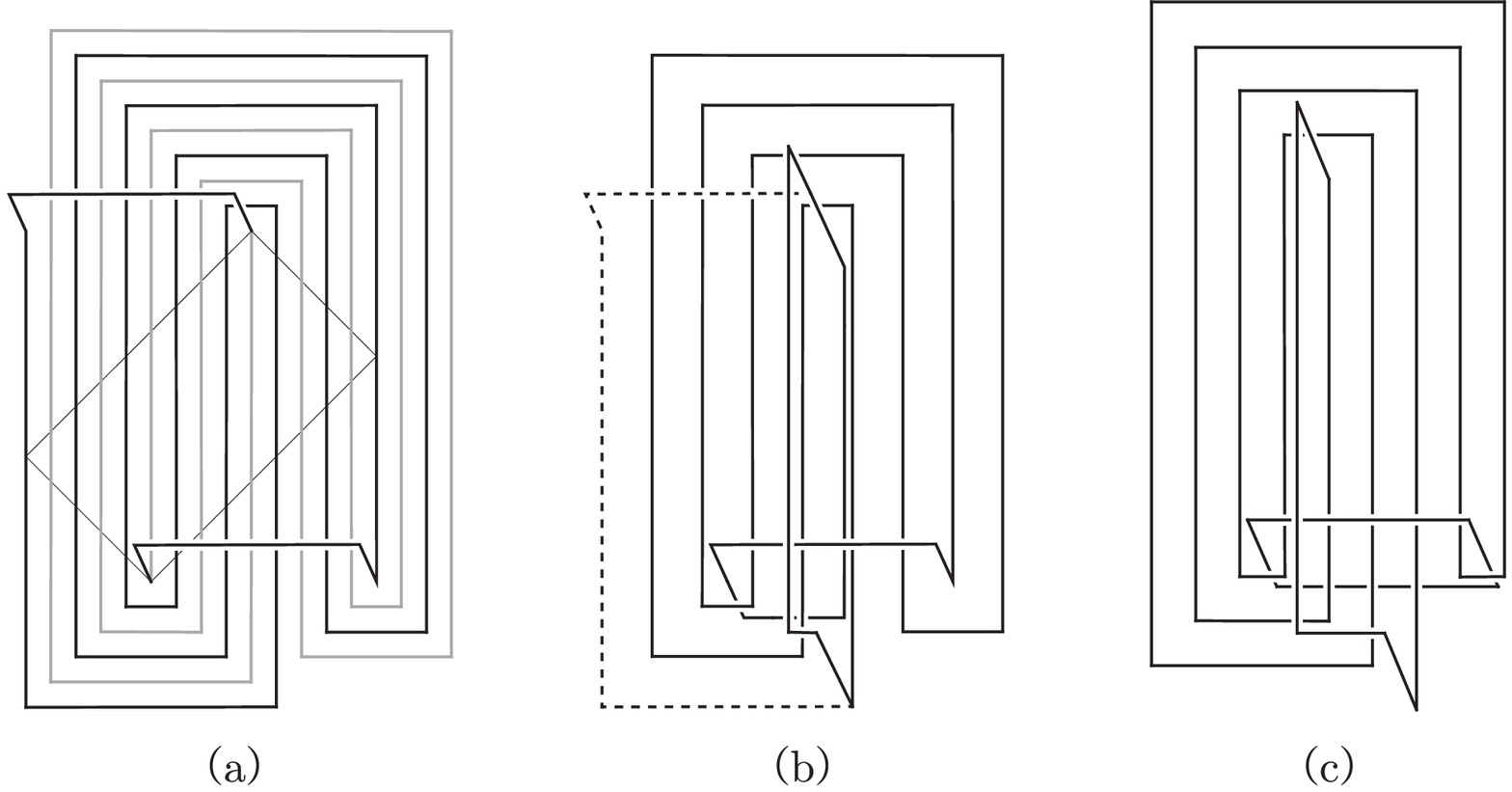}
\caption{A rational $\dfrac{p}{q}$-link in $\mathbb{Z}^3$}
\label{fig:regcir}
\end{figure}

Let $P_1$ be the path starting from $v_1$, and $P_2$ be the other.
Then the vertical line segments of $P_1$ have even $x$-coordinate, and the vertical line segments of $P_2$ have odd $x$-coordinate.
Since there are $2p$ vertical lines, each $P_1$ and $P_2$ has $p$ vertical line segments.
On the other hand, each path starts and finishes with vertical line segments.
Hence each $P_1$ and $P_2$ consists of $2p-1$ line segments.

Next, connect $v_1$ and $v'_1$ by an arc $l_1$ which consists of one horizontal line segment in $z$-level 1, and two $z$-sticks.
Similarly, connect $v_2$ and $v'_2$ by an arc $l_2$ as drawn in Figure~\ref{fig:regcir} (a).
Therefore, this 2-circuit presentation can be represented by $4p+4$ sticks in the cubic lattice.

Now we transform this presentation for reducing number of sticks to represent.
First, change the path which passes through $(p,p,0)$, $(p,p,1)$, $(0,p,1)$, $(0,p,0)$, $(0,-2q,0)$, and $(2q,-2q,0)$ to $(p,p,0)$, $(p,p,2)$, $(p,-2q,2)$, $(2q,-2q,2)$, and $(2q,-2q,0)$.
Then we can reduce the number of sticks by one.
Next, push down $P_2$ to $z$-level $-1$.
Then $P_2$ can be replaced by a path consisting at most two sticks as drawn in Figure~\ref{fig:regcir} (b).
Especially, if $L$ is a 2-component link, $P_2$ can be replaced by a stick because $v_2$ and $v'_2$ are the endpoints of $P_2$ as drawn in Figure~\ref{fig:regcir} (c).
Therefore for the rational $\dfrac{p}{q}$-link $L$, the upper bound of the lattice stick number of $L$ with exactly 4 $z$-sticks is $2p + 6$.
Furthermore it is $2p+5$ if $L$ is a 2-component link.

%
%===========================================
%
%\section{Conclusion}
%
%
%
% In this paper we constructed a lattice link which contain exactly 4 $z$-sticks by $2p+6$ sticks if $L$ is a knot and $2p+5$ sticks if $L$ is a 2-component link.
%1) lattice stick number of $3_1$ and $5_1$  are 12 and 16.
%
%2) compare with HNO result. 
%For $(p,1)$ and $(p,p-1)$, our upper bound is better
%

\end{document}